\documentclass[11pt]{amsart}
\usepackage[utf8]{inputenc}
\usepackage{fontenc}
\usepackage{amsfonts}
\usepackage{amssymb}
\usepackage{amsmath}
\usepackage{amsthm}
\usepackage{enumerate}
\usepackage{mathrsfs}
\usepackage{tikz}

\newcommand{\NN}{\mathbb{N}}

\newcommand{\er}{\mathbb{R}}
\newcommand{\ce}{\mathbb{C}}

\newtheorem{theorem}{Theorem}[section]
\newtheorem*{theorem*}{Theorem}
\newtheorem{proposition}[theorem]{Proposition}
\newtheorem*{proposition*}{Proposition}
\newtheorem{lemma}[theorem]{Lemma}
\newtheorem*{lemma*}{Lemma}
\newtheorem{corollary}[theorem]{Corollary}
\newtheorem*{corollary*}{Corollar}

\newtheorem*{fact*}{Fact}
\theoremstyle{definition}
\newtheorem{definition}[theorem]{Definition}
\newtheorem*{definition*}{Definition}
\newtheorem{claim}[theorem]{Claim}
\newtheorem*{claim*}{Claim}

\newtheorem*{conjecture*}{Conjecture}

\newtheorem{theoremi}{Theorem}

\theoremstyle{remark}

\newtheorem*{example*}{Example}
\newtheorem{remark}[theorem]{Remark}
\newtheorem*{remark*}{Remark}

\newtheorem*{note*}{Note}

\newtheorem*{question*}{Question}

\newcommand{\norm}[1]{\left\lVert #1 \right\rVert}

\DeclareMathOperator{\Id}{Id}

\DeclareMathOperator{\bd}{bd}

\DeclareMathOperator{\Homeo}{Homeo}

\DeclareMathOperator{\Aut}{Aut}

\usepackage{enumitem}

\begin{document}

\title{Nontrivial homeomorphisms of \v{C}ech-Stone remainders}%

\author[A. Vignati]{Alessandro Vignati}
\address[A. Vignati]{Department of Mathematics and Statistics\\
York University\\
4700 Keele Street\\
Toronto, Ontario\\ Canada, M3J
1P3\\
}
\email{ale.vignati@gmail.com}

\thanks{This research was partially completed during the author's fellowship at Institut Mittag-Leffler for the program on Classification of operator algebras: complexity, rigidity, and dynamics. The author would like to thanks the organizers for the support. The author is supported by a Susan Mann Scholarship at York University}

\date{\today}%
\begin{abstract}
We study the group of automorphisms of certain corona C*-algebras. As a corollary of a more general C*-algebraic result, we show that, under the Continuum Hypothesis, $\beta X\setminus X$ has nontrivial homeomorphisms, whenever $X$ is a noncompact locally compact metrizable manifold.
 \end{abstract}
 
\maketitle
\section{Introduction}
If $X$ is a locally compact, noncompact space we say that an homeomorphism $\phi$ of its \v{C}ech-Stone remainder $X^*=\beta X\setminus X$ is trivial if there is a continuous $f\colon X \to X$ with the property that $\phi=\beta f\restriction X^*$, where $\beta f$ is the unique continuous extension of $f$ to $\beta X$. If $X$ is Polish, there can be only $\mathfrak c=2^{\aleph_0}$ trivial homeomorphisms of $X^*$. In general, the task of finding nontrivial homeomorphisms of \v{C}ech-Stone remainders is a challenging one. In fact, for a Polish $X$ as above, the existence of nontrivial homeomorphisms of $X^*$ is conjectured to be independent from the usual axioms of set theory (ZFC, the Zermelo-Fraenkel scheme plus the Axiom of Choice). Under the Continuum Hypothesis (CH from now on) it is conjectured that there are nontrivial homeomorphisms of $X^*$, while under the assumption of different set theoretical axioms (like the Proper Forcing Axiom, PFA, or some of its consequences), it is conjectured that $\Homeo(X^*)$, the group of homeomorphisms of $X^*$, has only trivial elements.

The current state-of-the-art of the conjectures is as follow. 
Rudin's work (\cite{Rudin}) shows that, under CH, $\NN^*$ has $2^{\mathfrak c}$-many homeomorphisms, and hence, since $\mathfrak c<2^{\mathfrak c}$, nontrivial ones. The same result applies to $X^*$ whenever $X$ is locally compact noncompact Polish and zero dimensional (in this case, under CH, $X^*$ is homeomorphic to $\NN^*$ by Parovi\v{c}enko's Theorem). On the other hand, for such an $X$, the results in \cite{Shelah-Steprans.PFAA}, \cite{Velickovic.OCAA}, \cite{Farah.AQ}, and \cite{Farah-McKenney.ZD} showed that, under PFA (or some of its consequences), one can prove that $X^*$ has only trivial homeomorphisms.

At a current stage, even though partial progress has been made (e.g., \cite[Theorem 5.3]{Farah-Shelah.RCQ}), all the spaces $X$ for which PFA proves that all elements of $\Homeo(X^*)$ is trivial are zero dimensional.

On the side of the conjecture dealing with constructing nontrivial homeomorphisms of $X^*$ under CH, the first successful attempt to go beyond zero dimensionality was Yu's. He constructed, under CH, a nontrivial homeomorphism of $\er^*$ (see \cite[\S 9]{KP.STCC}). Currently, the sharpest results asserts that, if $X$ is locally compact, noncompact and Polish and assuming CH, there are nontrivial elements of $\Homeo(X^*)$ if $X$ has countably many clopen sets (\cite{Coskey-Farah}) or if $X$  is an increasing union of compact sets $K_n$ satisfying $\sup_n |K_n\cap (\overline {X\setminus K_n})|<\infty$ (\cite[Theorem 2.5]{Farah-Shelah.RCQ}). The latter result uses countable saturation of the C*-algebra $C_b(X)/C_0(X)$ (see \cite{farah2011countable} and \cite[Theorem 3.1]{Farah-Shelah.RCQ}) and provided a different proof of Yu's result on $\Homeo(\er^*)$.

For $n\geq 2$, it was unknown if one could have a nontrivial homeomorphism of $(\er^n)^*$. The following appealing consequence of Theorem \ref{thm:main} below settles this uncertainty.

\begin{theoremi}\label{thmi:A}
Assume CH and let $X$ be a locally compact metrizable noncompact manifold. Then $X^*$ has $2^{\mathfrak c}$-many homeomorphisms. In particular,  $X^*$ has nontrivial homeomorphisms.
\end{theoremi}

Even though our main result has a topological application, its proof combines both topological and C*-algebraic techniques and it is stated in terms of C*-algebras. The C*-algebraic approach is justified by the construction of the multiplier algebra of a nonunital C*-algebra, which is the noncommutative correspondent of the \v{C}ech-Stone compactification.
Given a nonunital C*-algebra $A$ one constructs its multiplier algebra, $\mathcal M(A)$, which is a nonseparable unital C*-algebra in which $A$ sits as an ideal in a universal way. The quotient $\mathcal M(A)/A$ is called the corona of $A$. Multipliers and coronas are objects carrying many properties of broad interest (see \cite{Lance.HM} and \cite{Pedersen.CC}).

In the commutative case, when $A=C_0(X)$ for some locally compact noncompact $X$, then $\mathcal M(A)=C_b(X)\cong C(\beta X)$ and $\mathcal M(A)/A\cong C(X^*)$. By duality, automorphisms of $\mathcal M(A)/A$ correspond bijectively to homeomorphisms of $X^*$. 

The interest on the influence of set theory in automorphisms of coronas of noncommutative C*-algebras was motivated by the seek of an outer automorphism of the Calkin algebra, the quotient of the bounded linear operators on $\ell^2(\NN)$ modulo the ideal of compact operators, see \cite{BDF}. It was proved, in \cite{Phillips-Weaver}, that under CH the Calkin algebra has $2^{\mathfrak c}$-automorphisms (and, therefore, outer ones), while Farah showed that under the Open Coloring Axiom (a consequence of PFA) the Calkin algebra has only inner automorphisms (see \cite{Farah.C}).

Conjecturally (see \cite[Conjectures 1.2 and 1.3]{Coskey-Farah}) if $A$ is a nonunital separable C*-algebra then CH implies that $\mathcal M(A)/A$ has $2^{\mathfrak c}$-many automorphisms, while PFA implies that the structure of $\Aut(\mathcal M(A)/A)$ is as rigid as possible (for the definition of triviality in the noncommutative case see \cite[Definition 1.1]{Coskey-Farah}). The conjectures have been verified for some classes of C*-algebras, even though a large amount of projections in $A$ is usually needed. For more on this, see \cite{Coskey-Farah}, \cite{Ghasemi.FDD}, \cite{Farah-Shelah.RCQ}, \cite{McKenney.UHF}  or the upcoming \cite{MKAV.AC}.

We are interested in C*-algebras of the form $C_0(X,A)$, the algebra of continuous function from some locally compact space $X$ to a C*-algebra $A$ which are vanishing at infinity. If $A$ is unital, the multiplier of $C_0(X,A)$ is (isomorphic to) $C_b(X,A)$, the algebra of bounded continuous functions from $X$ to $A$. We denote $C_b(X,A)/C_0(X,A)$ by $Q(X,A)$. Our main result, Theorem \ref{thm:main}, asserts that under the assumption of CH, if $X$ is a space satisfying a technical condition ensuring some sort of flexibility (e.g., a manifold, see Definition \ref{def:flexible}), and $A$ is a C*-algebra, then $Q(X,A)$ has $2^{\mathfrak c}$-many automorphisms. Note we do not ask for $A$ to be unital. On the other hand, if $A$ is nonunital, $Q(X,A)$ is not the corona of $C_0(X,A)$. In this case, we don't know whether a result similar to Theorem \ref{thm:main} applies to the corona of $C_0(X,A)$. For more on $\mathcal M(C_0(X,A))$ if $A$ is nonunital, see \cite{APT.Mult}.

The paper is structured as follows: Section \ref{sec:prel} contains preliminaries and notation, and section \ref{sec:mainthm} is dedicated to the proof of our main result. Lastly, we provide an example of a space $X$ for which it is not known whether $\Homeo(X^*)$ contains nontrivial elements, even under CH.

The author would like to thank Ilijas Farah for helpful suggestions and comments.
\section{Preliminaries and notation}\label{sec:prel}

Unless stated differently, $X$ is a locally compact noncompact Polish space with a metric $d$ inducing its topology and $A$ is a C*-algebra.
Given a closed $Y\subseteq X$ we say that $\phi\in \Homeo(Y)$ \emph{fixes the boundary} of $Y$ if, whenever $y\in \bd_X(Y)=Y\cap (\overline{X\setminus Y})$, then $\phi(y)=y$. We denote the set of all such homeomorphisms by $\Homeo_{\bd_X(Y)}(Y)$ (or $\Homeo_{\bd}(Y)$ if $X$ is clear from the context). When $\phi\in \Homeo_{\bd}(Y)$ it can be extended in a canonical way to $\tilde\phi\in\Homeo(X)$ by 
\[
\tilde\phi(x)=\begin{cases}\phi(x) & \text{ if }x\in Y\\
x&\text{otherwise.}
\end{cases}
\]
If $Y_n\subseteq X$, for $n\in \NN$, are closed and disjoint sets with the property that no compact subset of $X$ intersects infinitely many $Y_n$'s, we have that $Y=\bigcup Y_n$ is closed. If $\phi_n\in\Homeo_{\bd}(Y_n)$ then $\phi=\bigcup\phi_n\in\Homeo_{\bd}Y$ is well defined. In this situation we abuse of notation and say that $\tilde\phi$ as constructed above extends canonically $\{\phi_n\}$.

For a $\phi\in\Homeo(X)$ we will denote by $r(\phi)$ the \emph{radius} of $\phi$  as
\[
r(\phi)=\sup_{x\in X}d(x,\phi(x)).
\]
If $Y$ is compact and $\phi\in\Homeo_{\bd}(Y)$ we have that $r(\phi)<\infty$ and $r(\phi)$ is attained by some $y\in Y$. It can be easily verified that $r(\phi_0\phi_1)\leq r(\phi_0)+r(\phi_1)$ for $\phi_0,\phi_1\in\Homeo(X)$.

Note that every $\tilde\phi\in\Homeo(X)$ determines uniquely a $\psi\in\Aut(C_b(X,A))$, which induces a $\tilde\psi\in\Aut(Q(X,A))$. If $Y_n\subseteq X$ are disjoint closed sets with the property that no compact $Z\subseteq X$ intersects infinitely many of them and $\phi_n\in\Homeo_{\bd}(Y_n)$, we will abuse of notation and say that $\psi$ and $\tilde\psi$ are \emph{canonically determined} by $\{\phi_n\}$.

We are interested in a particular class of topological spaces:
\begin{definition}\label{def:flexible}
A locally compact noncompact Polish space $(X,d)$ is \emph{flexible} if there are disjoint sets $Y_n\subseteq X$ and $\phi_{n,m}\in\Homeo_{bd}(Y_n)$ with the following properties:
\begin{enumerate}[label=(\arabic*)]
\item\label{defin:cond1} every $Y_n$ is a compact subset of $X$ and there is no compact $Z\subseteq X$ that intersects infinitely many $Y_n$'s and
\item\label{defin:cond2} for all $n$, $r(\phi_{n,m})$ is a decreasing sequence tending to $0$ as $m\to\infty$, with $r(\phi_{n,m})\neq 0$ whenever $n,m\in\NN$.
\end{enumerate}
The sets $Y_n$ and the homeomorphisms $\phi_{n,m}$ are said to witness that $X$ is flexible.
\end{definition}
\begin{remark}\label{rem:manifolds}
 We don't know whether condition \ref{defin:cond2} is equivalent to having a sequence of disjoint $Y_n$'s satisfying \ref{defin:cond1} for which $\Homeo_{\bd}(Y_n)$ has a continuous path. This condition is clearly stronger than \ref{defin:cond2}. In fact, being $\Homeo_{bd}(Y_n)$ a group, if it contains a path, then there is a path $a(t)\subseteq\Homeo_{\bd}(Y_n)$ with $a(0)=Id$ and $a(t)\neq a(0)$ if $t\neq 0$. By continuity, if a path exists, it can be chosen so that $s<t$ implies $r(a(s))<r(a(t))$. Since any closed ball in $\er^n$ has this property, a typical example of a flexible space is a manifold. 

We should also note that if $X$ is a locally compact Polish space for which there is a close discrete sequence $x_n$ and a sequence of open sets $U_n$ with $U_i\cap U_j=\emptyset$ if $i\neq j$, $x_n\in U_n$, and such that each $U_n$ is a manifold, then $X$ is flexible. In particular, if $X$ has a connected component which is a noncompact Polish manifold, then $X$ is flexible.

Lastly, if $X$ is flexible and $Y$ has a compact clopen $Z$, then $\bd(Y_n\times Z)=\bd(Y_n)\times Z$, therefore $Z_n=Y_n\times Z$ and $\rho_{m,n}=\psi_{n,m}\times id$ witness the flexibility of $X\times Y$. In particular, if $Y$ is compact, $X\times Y$ is flexible.
\end{remark}

By $\NN^{\NN}$ we denote the set of all sequences of natural numbers, where $f(n)>0$ for all $n$. If $f_1,f_2\in\NN^{\NN}$ we write $f_1\leq^* f_2$ if 
\[
\forall^\infty n  (f_1(n)\leq f_2(n)).\footnote{The shortening $\forall^\infty n$ is for $\exists n_0\forall n\geq n_0$. Equally $\exists^\infty n$ is for \emph{there are infinitely many $n$}.}
\]
If $Y_n\subseteq X$ are compact sets with the property that no compact $Z\subseteq X$ intersects infinitely many $Y_n$, we can associate to every $f\in\NN^{\NN}$ a subalgebra of $C_b(X,A)$ as
\begin{eqnarray*}
D_f(X,A,Y_n)=\{g\in C_b(X,A)\mid &&\forall \epsilon>0\forall^\infty n\forall x,y\in Y_n \\
&&(d(x,y)<\frac{1}{f(n)}\Rightarrow \norm{g(x)-g(y)}<\epsilon)\}.
\end{eqnarray*}
We denote by $C_f(X,A,Y_n)$ the image of $D_f(X,A,Y_n)$ under the quotient map $\pi\colon C_b(X,A)\to Q(X,A)$. (If $X$, $Y_n$ and $A$ are clear from the context, we simply write $D_f$ and $C_f$.

The following proposition clarifies the structure of the $D_f$'s and the $C_f$'s.

\begin{proposition}\label{prop:buildingblocks}
Let $(X,d)$ be a locally compact noncompact Polish space and $A$ be a C*-algebra. Let $Y_n\subset X$ be infinite  compact disjoint sets such that no compact subset of $X$ intersects infinitely many $Y_n$'s. Then:
\begin{enumerate}[label=(\arabic*)]
\item\label{prop.bb1} For all $f\in\NN^\NN$ we have that $D_f$ is a C*-subalgebra of $C_b(X,A)$. If $A$ is unital, so is $D_f$;
\item\label{prop.bb2} if $f_1\leq^* f_2$ then $C_{f_1}\subseteq C_{f_2}$;
\item\label{prop.bb3} $C_b(X,A)=\bigcup_{f\in\NN^{\NN}}D_f$;
\item\label{prop.bb4} for all $f\colon \NN\to\NN$ there is $g\in C_b(X,A)$ such that $\pi(g)\notin C_f$.
\end{enumerate}
\end{proposition}
\begin{proof}
(1) and (2) follow directly from the definition of $D_f$ and $C_f$. For (3), take $g\in C_b(X,A)$. Since each $Y_n$ is compact and metric we have that $g\restriction Y_n$ is uniformly continuous. In particular there is $\delta_n>0$ such that $d(x,y)<\delta_n$ implies $\norm{g(x)-g(y)}<2^{-n}$ for all $x,y\in Y_n$. Fix $m_n$ such that $\frac{1}{m_n}<\delta_n$ and let $f(n)=m_n$. Then $g\in D_f$.

For (4), fix $f\in\NN^\NN$ and $x_n\neq y_n\in Y_n$ with $d(y_n,z_n)<\frac{1}{f(n)}$. Since no compact set intersects infinitely many $Y_n$'s, both $Y'=\{y_n\}_n$ and $Z'=\{z_n\}_n$ are closed in $X$. Pick any $a\in A$ with $\norm{a}=1$ and let $g$ be a bounded continuous function such that $g(Y')=0$ and $g(Z')=a$. It is easy to see that $g\notin C_f$.
\end{proof}
The following Lemma represents the connections between the filtration we obtained and an automorphism of $Q(X,A)$.
\begin{lemma}\label{lem:agreeing}
Let $(X,d)$, $A$, and $Y_n$ be as in Proposition \ref{prop:buildingblocks} and suppose that $\phi_n\in \Homeo_{\bd}(Y_n)$. Let $\tilde\phi\in\Homeo(X)$ and $\tilde\psi\in\Aut(Q(X,A))$ be canonically determined by $\{\phi_n\}$ and $f\in\NN^{\NN}$. Then:
\begin{enumerate}[label=(\arabic*)]
\item\label{lem:agreeingc1} if there are $k,n_0$ such that for all $n\geq n_0$ 
\[
r(\phi_n)\leq\frac{k}{f(n)}
\]
 we have that $\tilde\psi(g)=g$ for all $g\in C_f$;
\item\label{lem:agreeingc2} if for infinitely many $n$ we have that 
\[
r(\phi_n)\geq\frac{n}{f(n)}.
\]
 then there is $g\in C_f$ such that $\tilde\psi(g)\neq g$.
\end{enumerate}
\end{lemma}
\begin{proof}
Note that, if $g\in C_b(X,A)$ and $\tilde\psi$ is as above, we have $\tilde\psi(g)=g$ if and only if $g-\psi(g)\in C_0(X,A)$ where $\psi\in\Aut(C_b(X,A))$ is canonically determined by $\{\phi_n\}$.

To prove (1), let $k,n_0$ as above. Fix $\epsilon>0$ and $n_1>n_0$ such that whenever $n\geq n_1$ we have that $\frac{k}{f(n)}<\epsilon$ and if $x,y\in Y_n$ with $d(x,y)<\frac{1}{f(n)}$ then $\norm{g(x)-g(y)}<\epsilon/k$. Such an $n_1$ can be found, since $g\in D_f$. 
Let now $x\notin \bigcup_{i\leq n_1}Y_i$. Since $g(x)-\psi(g(x))=g(x)-g(\tilde\phi(x))$, if $x\notin\bigcup Y_n$ we have $\tilde\phi(x)=x$ and so $g(x)-\psi(g)(x)=0$. 
If $x\in Y_n$ for $n\geq n_1$,we have $d(x,\phi_n(x))<r(\phi_n)\leq\frac{k}{f(n)}$ and by our choice of $n_1$,
\[
\norm{g(x)-\psi(g)(x)}=\norm{g(x)-g(\phi_n(x))}\leq\epsilon.
\]
Since $\bigcup_{i\leq n_1}Y_i$ is compact, we have that $g-\psi(g)\in C_0(X,A)$, and (1) follows.

For (2), let $k(n)$ be a sequence of natural numbers such that 
\[
r(\phi_{k(n)})\geq\frac{k(n)}{f(k(n))}.
\]
We will construct $h\in D_f$ and show that $h-\psi(h)\notin C_0(X,A)$. Fix some $a\in A$ with $\norm{a}=1$.
If $m\neq k(n)$ for all $n$, set $h(Y_m)=0$. If $m=k(n)$, let $r=r(\phi_m)$ and pick $x_0=x_0(m)$ such that $d(x_0,\phi_m(x_0))=r$. Set $x_1=x_1(m)=\phi_m(x_0)$ and, for $i=0,1$, let 
\[
Z_i=\{z\in Y_m\mid d(z,x_i)\leq r/2\}.
\]
If $z\in Z_0$ define 
\[
h(z)=(\frac{d(z,x_0)}{r})a
\]
and if $z\in Z_1$ let
\[
h(z)=(1-\frac{d(z,x_1)}{r})a,
\]
while for $z\in Y_m\setminus(Z_0\cup Z_1)$ let $h(z)=\frac{a}{2}$. Let $h'\in C_b(X,A)$ be any function such that $h'(x)=h(x)$ whenever $x\in\bigcup Y_i$. Note that we have that $h'\in D_f$, as this only depends on its values on $\bigcup Y_i$. We want to show that $h'-\psi(h')\notin C_0(X,A)$. To see this, note that if $m=k(n)$ for some $n$ we have 
\[
h'(x_0(m))=0\text{ and }\psi(h')(x_0(m))=h'(\psi_m(x_0(m)))=h'(x_1(m))=a.
\] 
Since $\{x_0(m)\}_{m\in\NN}$ is not contained in any compact subsets of $X$ we have the thesis.
\end{proof}

We are ready to introduce our main concept.
\begin{definition}\label{coherent}
Let $(X,d)$, $Y_n$ and $A$ be as in Proposition \ref{prop:buildingblocks} and $\kappa$ be uncountable. Let $\{f_\alpha\}_{\alpha<\kappa}\subseteq\NN^{\NN}$ be a $\leq^*$-increasing sequence of functions and $\{\phi_n^\alpha\}_{\alpha<\kappa}$ be such that for all $\alpha$ and $n$, 
\[
\phi_n^\alpha\in\Homeo_{\bd}(Y_n).
\]
$\{\phi_n^\alpha\}$ is said \emph{coherent with respect to }$\{f_\alpha\}$ if 
\[
\alpha<\beta\Rightarrow \exists k \forall^\infty n (r(\phi_n^\alpha(\phi_n^{\beta})^{-1})\leq\frac{k}{f_\alpha(n)}).
\]

If $\gamma$ is countable, $\{\phi_n^\alpha\}_{\alpha\leq\gamma}$ is coherent w.r.t.  $\{f_\alpha\}_{\alpha\leq\gamma}\subseteq\NN^{\NN}$ if for all $\alpha<\beta\leq\gamma$ we have that $ \exists k \forall^\infty n (r(\phi_n^\alpha(\phi_n^{\beta})^{-1})\leq\frac{k}{f_\alpha(n)})$.
\end{definition}
\begin{remark}
Definition \ref{coherent} is stated in great generality. We don't ask for the sequence $\{f_\alpha\}_{\alpha<\kappa}$ to have particular properties (e.g., being cofinal) or for the space $X$ to be flexible, even though Definition \ref{coherent} will be used in such context.

Note that if $\{\phi_n^\alpha\}_{\alpha<\omega_1}$ is such, that for all $\gamma<\omega_1$, $\{\phi_n^\alpha\}_{\alpha\leq\gamma}$ is coherent w.r.t. $\{f_\alpha\}_{\alpha\leq\gamma}$ then $\{\phi_n^\alpha\}_{\alpha<\omega_1}$ is coherent w.r.t. $\{f_\alpha\}_{\alpha<\omega_1}$.
\end{remark}

Recalling that $\mathfrak d$ denotes the smallest cardinality of a $\leq^*$-cofinal family in $\NN^\NN$, we say that a $\leq^*$ increasing and cofinal sequence $\{f_\alpha\}_{\alpha<\kappa}\subseteq\NN^{\NN}$, for some $\kappa>\mathfrak d$, is \emph{fast} if for all $\alpha$ and $n$, 
\[
 n f_\alpha(n)\leq f_{\alpha+1}(n).
\]
If $\{f_\alpha\}_{\alpha<\mathfrak d}$ is fast, the same argument as in Proposition~\ref{prop:buildingblocks} shows that
\[
Q(X,A)=\bigcup_\alpha C_{f_\alpha}.
\]

The following lemma is going to be key for our construction. Its proof follows almost immediately from the definitions above, but we sketch it for convenience.

\begin{lemma}\label{lemma:coherent->uniqueness}
Let $(X,d)$, $A$ and $Y_n$ be fixed as in Proposition~\ref{prop:buildingblocks}. Let $\{f_\alpha\}$ be a fast sequence and suppose that $\{\phi_n^\alpha\}$ is a coherent sequence w.r.t. $\{f_\alpha\}$. Let $\tilde\psi_\alpha\in\Aut(Q(X,A))$ be canonically determined by $\{\phi_n^\alpha\}_{n}$. Then there is a unique $\tilde\Psi\in\Aut(Q(X,A))$ with the property that 
\[
\tilde\Psi(g)=\tilde\psi_\alpha(g), \,\,\, g\in C_{f_\alpha}.
\]
\end{lemma}
\begin{proof}
We define $\tilde\Psi(g)=\tilde\psi_\alpha(g)$ for $g\in C_{f_\alpha}$. If $\alpha<\beta$,  we define $\tilde\psi_{\alpha\beta}=\tilde\psi_\alpha(\tilde\psi_\beta)^{-1}$. As $\tilde\psi_{\alpha\beta}$ is canonically determined by $\{\psi_{n}^\alpha(\phi_n^{\beta})^{-1}\}_n$, and by coherence there are $k,n_0\in\NN$ such that whenever $n>n_0$ we have 
\[
r(\phi_n^\alpha(\phi_n^{\beta})^{-1})<\frac{k}{f_\alpha(n)}.
\]
 By condition (1) of Lemma \ref{lem:agreeing} we therefore have that $\tilde\psi_{\alpha\beta}(g)=g$ whenever $g\in C_{f_\alpha}$, and in this case $\tilde\psi_{\alpha}(g)=\tilde\psi_\beta(g)$, so $\tilde\Psi$ is a well defined morphisms of $Q(X,A)$ into itself.
Let $\tilde\psi'_\alpha\in\Aut(Q(X,A))$ be canonically determined by $\{(\psi_n^\alpha)^{-1}\}_n$. Since $\{\phi_n^\alpha\}$ is coherent w.r.t $\{f_\alpha\}$, so is $\{(\phi_n^{\alpha)^{-1}}\}$. In particular, if we let $\tilde\Psi'$ defined by $\tilde\Psi'(g)=\tilde\psi'_\alpha(g)$ for $g\in C_{f_\alpha}$, we have that $\tilde\psi'$ is a well-defined morphisms from $Q(X,A)$ into itself, with the property that $\tilde\Psi'\tilde\Psi=\tilde\Psi\tilde\Psi'=Id$, hence $\tilde\Psi$ is an automorphism.
This concludes the proof.
\end{proof}
\section{The construction}\label{sec:mainthm}
This section is dedicated to prove our main result, that is the following:
\begin{theorem}\label{thm:main}
Let $X$ be flexible and $A$ be a C*-algebra. Suppose that $\mathfrak d=\omega_1$ and $2^{\aleph_0}<2^{\aleph_1}$. Then $Q(X,A)$ has $2^{\aleph_1}$-many automorphisms. In particular, under CH, there are $2^{\mathfrak c}$-many automorphisms of $Q(X,A)$.
\end{theorem}
\begin{proof}
Fix $d$, $Y_n$ and $\phi_{n,m}\in\Homeo_{\bd}Y_n$ witnessing that $X$ is flexible.

We have to give a technical restriction (see Remark \ref{rem:technicalities}) on the kind of elements of $\NN^{\NN}$ we are allowed to use. This restriction depends strongly on the choice of $d$, on the witnesses $Y_n$ and on the $\phi_{n,m}$'s. We define
\[
A_n=\{k\in \NN\mid\exists m (r(\phi_{n,m})\in [1/(k+1),1/k])\}
\]
As $r(\phi_{n,m})\to 0$ for $m\to \infty$, $A_n$ is always infinite. We define 
\[
\NN^{\NN}(X)=\{f\in\NN^{\NN}\mid f(n)\in A_n\}\subseteq\NN^{\NN}
\]
Since each $A_n$ is infinite, $\NN^{\NN}(X)$ is cofinal in $\NN^{\NN}$.

As $\mathfrak d=\omega_1$, we can fix a fast sequence $\{f_\alpha\}_{\alpha\in\omega_1}\subseteq\NN^{\NN}(X)$. Let $C_\alpha:=C_{f_\alpha}$.
Finally fix, for each limit ordinal $\beta<\omega_1$, a sequence $\alpha_{\beta,n}$ that is strictly increasing and cofinal in $\beta$.

We will make use of Lemma \ref{lemma:coherent->uniqueness} and construct, for each $p\in 2^{\omega_1}$, a sequence $\phi_n^\alpha(p)$ that is coherent w.r.t. $\{f_\alpha\}$.  For simplicity we write $\phi_n^\alpha$ for $\phi_n^\alpha(p)$.
Let $\phi_n^0=\Id$. Once $\phi_n^\alpha$ has been constructed, let 
\[
\phi_n^{\alpha+1}=\phi_{n,m}\phi_n^\alpha,\,\,\, \text{ if }p(\alpha)=1,
\]
where $m$ is the smallest integer such that $r(\phi_{n,m})\in[\frac{1}{f_\alpha(n)+1},\frac{1}{f_\alpha(n)}]$, and $\phi_n^{\alpha+1}=\phi_n^\alpha$ otherwise.
\begin{claim}
If $\{\phi_n^\gamma\}_{\gamma\leq\alpha}$ is coherent w.r.t. $\{f_\gamma\}_{\gamma\leq\alpha}$ then $\{\phi_n^\gamma\}_{\gamma\leq\alpha+1}$ is coherent w.r.t. $\{f_\gamma\}_{\gamma\leq\alpha+1}$.
\end{claim}
\begin{proof}
We want to show that whenever $\gamma<\alpha$ there is $k$ such that 
\[
\forall^\infty n (r(\phi_n^\gamma(\phi_n^{\alpha+1})^{-1})\leq\frac{k}{f_\gamma(n)}).
\]
If $p(\alpha)=0$ this is clear, so suppose that $p(\alpha)=1$.

Note that 
\[
\phi_n^\gamma(\phi_n^{\alpha+1})^{-1}=\phi_n^\gamma(\phi_n^{\alpha})^{-1}\phi_{n,m}^{-1}
\]
where $m$ was chosen as above, and so
\[
r(\phi_n^\gamma(\phi_n^{\alpha+1})^{-1})\leq r(\phi_n^\gamma(\phi_n^{\alpha})^{-1})+r(\phi_{n,m})\leq  \frac{k}{f_\gamma(n)}+\frac{1}{f_\alpha(n)}
\]
for some $k$ (and eventually after a certain $n_0$). Since $f_\alpha(n)\geq f_\gamma(n)$ (again, eventually after a certain $n_1$), the conclusion follow.
\end{proof}
We are left with the limit step. Suppose then that $\phi_n^\alpha$ has be defined whenever $\alpha<\beta$. For shortness, let $\alpha_i=\alpha_{i,\beta}$.
\begin{claim}
For all $i\in\NN$ there is $k_i$ such that whenever $j\geq i$ there exists $n_{i,j}$ such that
\[
r(\phi_n^{\alpha_i}(\phi_n^{\alpha_j})^{-1}))\leq\frac{k_i}{f_{\alpha_i(n)}},
\]
whenever $n\geq n_{i,j}$
\end{claim}
\begin{proof}
Fix $i\in\NN$. By coherence there are $\bar k<\bar n$ such that whenever $n\geq\bar n$ we have 
\[
r(\phi_n^{\alpha_i}(\phi_n^{\alpha_{i+1}})^{-1}))<\frac{\bar k}{f_{\alpha_i}(n)}.
\]
Let $j>i$ and $n'(j)>k'(j)>\bar n$ such that if $n\geq n'(j)$ then 
\[
r(\phi_n^{\alpha_{i+1}}(\phi_n^{\alpha_j})^{-1}))<\frac{k'(j)}{f_{\alpha_{i+1}}(n)}
\]
and
\[
f_{\alpha_{i+1}}(n)\geq nf_{\alpha_i}(n).
\]
Fix $k_i=\bar k+1$ and $n_{i,j}=n'(j)$. Then for $n\geq n_{i,j}$
\begin{eqnarray*}
r(\phi_n^{\alpha_i}(\phi_n^{\alpha_j})^{-1}))&\leq& r(\phi_n^{\alpha_i}(\phi_n^{\alpha_{i+1}})^{-1}))+r(\phi_n^{\alpha_{i+1}}(\phi_n^{\alpha_j})^{-1}))\leq \frac{k'(j)}{f_{\alpha_{i+1}}(n)}+\frac{\bar k}{f_{\alpha_i}(n)}\\
&\leq& \frac{n}{f_{\alpha_{i+1}}(n)}+\frac{\bar k}{f_{\alpha_i}(n)}\leq \frac{k_i}{f_{\alpha_i}(n)}
\end{eqnarray*}
\end{proof}
Fix an sequence of $k_i$ as provided by the claim. Let $m_0=0$ and $m_{i+1}$ be the least natural above $m_i$ such that if $n\geq m_i$ and $j> i\geq l$ then 
\[
r(\phi_n^{\alpha_l}(\phi_n^{\alpha_j})^{-1})<\frac{k_l}{f_{\alpha_l}(n)}.
\]
Defining $\phi_n^\beta=\phi_n^{\alpha_i}$ whenever $n\in [m_{i-1},m_i)$, we have that coherence is preserved, that is, $\{\phi_n^\alpha\}_{\alpha\leq\beta}$ is coherent w.r.t. $\{f_\alpha\}_{\alpha\leq\beta}$. We just proved that we can define $\phi_n^\alpha$ for every countable ordinal.

By the remark following Definition \ref{coherent} we have that the sequence $\{\phi_n^\alpha\}_{\alpha<\omega_1}$ it is coherent w.r.t. $\{f_\alpha\}_{\alpha<\omega_1}$. By Lemma~\ref{lemma:coherent->uniqueness} there is a unique $\tilde\Psi=\tilde\Psi_p\in\Aut(Q(X,A))$ determined by $\{\phi_{n}^\alpha(p)\}_{\alpha<\omega_1}$. 

To conclude the proof, we claim that if $p\neq q$ we have $\tilde\Phi_p\neq\tilde\Phi_q$.
Let $\alpha$ be minimum such that $p(\alpha)\neq q(\alpha)$, and suppose $p(\alpha)=1$. Then 
\[
\phi_n^{\alpha}(p)=\phi_n^\alpha(q)
\]
so
\[
\phi_n^{\alpha+1}(p)=\phi_{n,m}\phi_n^{\alpha}(p)=\phi_{n,m}\phi_n^{\alpha}(q)
\]
where $m$ is the smallest integer for which $r(\phi_{n,m})\in [1/(f_\alpha(n)+1),1/f_\alpha(n)]$. In particular, eventually after a certain $n_0$,
\[
r(\phi_n^{\alpha+1}(q)\phi_n^{\alpha+1}(p)^{-1})=r(\phi_{n,m})\geq \frac{1}{f_\alpha(n)+1}\geq\frac{n}{f_{\alpha+1}(n)}.
\]
By Lemma \ref{lem:agreeing}, if $\tilde\psi_{\alpha+1}(p)\in\Aut(Q(X,A))$ is determined by $\phi_n^{\alpha+1}(p)$,  there is $g\in C_{\alpha+1}$ such that 
\[
\tilde\psi_{\alpha+1}(p)(g)\neq\tilde\psi_{\alpha+1}(q)(g).
\]
As Lemma \ref{lemma:coherent->uniqueness} states that 
\[
\tilde\Psi_p(g)=\tilde\psi_{\alpha+1}(p)(g)
\]
whenever $g\in C_{\alpha+1}$, we have that
\[
\tilde\Psi_p(g)=\tilde\psi_{\alpha+1}(p)(g)\neq \tilde\psi_{\alpha+1}(q)(g)=\tilde\Psi_q(g)
\]
\end{proof}
\begin{remark}\label{rem:technicalities}
The requirement of using $\NN^{\NN}(X)$ instead of $\NN^{\NN}$ is purely technical. Following Remark \ref{rem:manifolds}, if it is possible to choose $Y_n$ so that $\Homeo_{\bd}(Y_n)$ has a path (e.g., if $X$ is a manifold) then we can pick $\{\phi_{n,m}\}$ in order to have $A_n=\NN$ (eventually truncating a finite set).
\end{remark}

As promised, we are ready to give a proof of Theorem \ref{thmi:A}, which in particular shows that, whenever $n\geq 1$, $(\er^n)^*$ has plenty of nontrivial homeomorphisms under CH.
\begin{corollary}\label{cor1}
Assume CH. Let $X$ be a locally compact noncompact metrizable manifold. Then there are $2^{\mathfrak c}$-many nontrivial homeomorphisms of $X^*$. Suppose moreover that $Y$ is a locally compact space with a compact connected component. Then $(X\times Y)^*$ has $2^{\mathfrak c}$-many nontrivial homeomorphisms.
\end{corollary}
\begin{proof}
Manifolds are flexible thanks to remark \ref{rem:manifolds}, and homeomorphisms of $X^*$ correspond to automorphisms of $Q(X,\ce)$. Since there can be only $\mathfrak c$-many trivial homeomorphisms of $X^*$, the first assertion is proved. The second assertion follows similarly from remark \ref{rem:manifolds}, as if $X$ is flexible and $Y$ has a compact connected component, then $X\times Y$ is flexible.
\end{proof}
Even though Theorem \ref{thm:main} doesn't apply to the corona of $C_0(X,A)$ whenever $A$ is nonunital, we can stil say something in a particular case. Along the same lines as in the proof of Corollary \ref{cor1}, if $A$ is a C*-algebra that has a nonzero central projection\footnote{A projection $p$ is central if $pa=ap$ for all $a\in A$.} then it is possible to prove that under CH the corona of $C_0(X,A)$ has $2^{\mathfrak c}$-many automorphisms.

To conclude the paper, we show the existence of a one dimensional space $X$ for which it is still unknown whether CH implies the existence of a nontrivial element of $\Homeo(X^*)$. In fact, this space it is not flexible (according to Definition \ref{def:flexible}), it is connected (so the main result \cite{Coskey-Farah} doesn't apply), and it does not have an increasing sequence of compact subsets $K_n$ for which $\sup_n|K_n\cap (\overline{X\setminus K_n})|<\infty$ (and therefore it doesn't satisfy the hypothesis of \cite[Theorem 2.5]{Farah-Shelah.RCQ}). The space $X$ is based on a space described by Kuperberg, which appeared in the introduction of \cite{Phillips-Weaver}. 

The construction of $X$ goes as follows: take a copy of interval $a(t)$, $t\in [0,1]$ and let $x_0=a(0)$. At the midpoint of $a$, we attach a copy of the interval. We now have three copies of the interval attached to each other. At the midpoint of the first one, we attach two copies of the interval, at the midpoint of the second one we attach three of them, and four to the third. We order the new midpoints and attach five intervals to the first one, six to the second, and so forth. We repeat this construction infinitely many times, making sure that the length of the new intervals attached is short enough to have the whole construction contained in a bounded set and such that the intervals attached are short enough not to intersect. Let $Y$ be the closure of this iterated construction. For $n\geq 1$, let $y_n=a(\frac{1}{2^n})$, $a_n$ be one of the intervals attached to $y_n$, and $x_n$ be the endpoint of $a_n$ not belonging to $a$. Let $X=Y\setminus\{x_n\}_{n\geq 0}$. Since $x_n\to x_0$, as the intervals attached to $a(\frac{1}{2^n})$ become shorter and shorter when $n\to\infty$, $X$ is locally compact. Also $X$ has no homeomorphisms (for the same reason as in \cite{Phillips-Weaver}) and it is connected therefore $X$ is not flexible and it doesn't satisfy \cite[Hypothesis 4.1]{Coskey-Farah}.

We are left to show that we cannot apply \cite[Theorem 2.5]{Farah-Shelah.RCQ}, that is, we show that if $X=\bigcup K_n$ for some compact sets $K_n$, then $\sup_n|K_n\cap(\overline{X\cap K_n})|=\infty$. Note that $a_k\setminus\{x_k\}$ cannot be contained in a compact set of $X$ and all $a_k$'s are disjoint. In particular, if $K$ is compact and $K\cap a_k\neq\emptyset$ then there is $y\in K\cap a_k$ such that $y\in\overline{X\setminus K}$. If $\bigcup K_n=X$ for some compact sets $K_n$, for all $k$ there is $n$ such that $K_n\cap a_i\neq\emptyset$ for all $i\leq k$. Therefore $|K_n\cap (\overline{X\setminus K_n})|\geq k$. 

\bibliographystyle{amsalpha}
\bibliography{CHlibrary}

\providecommand{\bysame}{\leavevmode\hbox to3em{\hrulefill}\thinspace}
\providecommand{\MR}{\relax\ifhmode\unskip\space\fi MR }
\providecommand{\MRhref}[2]{%
  \href{http://www.ams.org/mathscinet-getitem?mr=#1}{#2}
}
\providecommand{\href}[2]{#2}
\begin{thebibliography}{APT73}

\bibitem[APT73]{APT.Mult}
C.~A. Akemann, G.~K. Pedersen, and J.~Tomiyama, \emph{Multipliers of
  {$C\sp*$}-algebras}, J. Functional Analysis \textbf{13} (1973), 277--301.

\bibitem[BDF77]{BDF}
L.~G. Brown, R.~G. Douglas, and P.~A. Fillmore, \emph{Extensions of
  {$C\sp*$}-algebras and {$K$}-homology}, Ann. of Math. (2) \textbf{105}
  (1977), no.~2, 265--324.

\bibitem[CF14]{Coskey-Farah}
S.~Coskey and I.~Farah, \emph{Automorphisms of corona algebras, and group
  cohomology}, Trans. Amer. Math. Soc. \textbf{366} (2014), no.~7, 3611--3630.

\bibitem[Far00]{Farah.AQ}
I.~Farah, \emph{Analytic quotients: theory of liftings for quotients over
  analytic ideals on the integers}, Mem. Amer. Math. Soc. \textbf{148} (2000),
  no.~702, xvi+177.

\bibitem[Far11]{Farah.C}
\bysame, \emph{All automorphisms of the {C}alkin algebra are inner}, Ann. of
  Math. (2) \textbf{173} (2011), no.~2, 619--661.

\bibitem[FH13]{farah2011countable}
I.~Farah and B.~Hart, \emph{Countable saturation of corona algebras}, C. R.
  Math. Acad. Sci. Soc. R. Can. \textbf{35} (2013), no.~2, 35--56.

\bibitem[FM12]{Farah-McKenney.ZD}
I.~Farah and P.~McKenney, \emph{Homeomorphisms of {C}ech-stone remainders: the
  zero-dimensional case}, arXiv:1211.4765 [math.LO], November 2012.

\bibitem[FS14]{Farah-Shelah.RCQ}
I.~Farah and S.~Shelah, \emph{Rigidity of continuous quotients},
  arXiv:1401.6689 [math.LO], 2014.

\bibitem[Gha14]{Ghasemi.FDD}
S.~Ghasemi, \emph{Isomorphisms of quotients of fdd-algebras}, arXiv:1310.1353,
  2014.

\bibitem[Har92]{KP.STCC}
K.~P. Hart, \emph{The \v {C}ech-{S}tone compactification of the real line},
  Recent progress in general topology ({P}rague, 1991), North-Holland,
  Amsterdam, 1992, pp.~317--352.

\bibitem[Lan95]{Lance.HM}
E.~C. Lance, \emph{Hilbert {$C\sp *$}-modules}, London Mathematical Society
  Lecture Note Series, vol. 210, Cambridge University Press, Cambridge, 1995, A
  toolkit for operator algebraists.

\bibitem[McK13]{McKenney.UHF}
P.~McKenney, \emph{Reduced products of {UHF} algebras}, arXiv:1303.5037
  [math.LO], March 2013.

\bibitem[MV]{MKAV.AC}
P.~McKenney and A.~Vignati, \emph{Forcing axioms and coronas of nuclear
  {C}*-algebras}, in preparation.

\bibitem[Ped90]{Pedersen.CC}
G.~K. Pedersen, \emph{The corona construction}, Operator {T}heory:
  {P}roceedings of the 1988 {GPOTS}-{W}abash {C}onference ({I}ndianapolis,
  {IN}, 1988), Pitman Res. Notes Math. Ser., vol. 225, Longman Sci. Tech.,
  Harlow, 1990, pp.~49--92.

\bibitem[PW07]{Phillips-Weaver}
N.~Christopher Phillips and N.~Weaver, \emph{The {C}alkin algebra has outer
  automorphisms}, Duke Math. J. \textbf{139} (2007), no.~1, 185--202.

\bibitem[Rud56]{Rudin}
W.~Rudin, \emph{Homogeneity problems in the theory of {\v{{c}}}ech
  compactifications}, Duke Math. J. \textbf{23} (1956), 409--419.

\bibitem[SS88]{Shelah-Steprans.PFAA}
S.~Shelah and J.~Stepr{\=a}ns, \emph{P{FA} implies all automorphisms are
  trivial}, Proc. Amer. Math. Soc. \textbf{104} (1988), no.~4, 1220--1225.

\bibitem[Vel93]{Velickovic.OCAA}
B.~Veli{\v{c}}kovi{\'c}, \emph{{$\textrm{OCA}$} and automorphisms of
  {$\mathcal{P}(\omega)/\textrm{fin}$}}, Topology Appl. \textbf{49} (1993),
  no.~1, 1--13.

\end{thebibliography}
\end{document}